\documentclass[letterpaper, 10 pt, conference]{ieeeconf}

    \let\labelindent\relax

\IEEEoverridecommandlockouts                              

\usepackage{amsmath}
\usepackage{amsthm}
\usepackage{amssymb}
\usepackage{amsfonts}
\usepackage{bbm}
\usepackage{bm}
\usepackage{mathtools}
\usepackage{graphicx}
\usepackage{epsfig,epsf,url,enumitem}
\usepackage{accents}
\usepackage[ruled,vlined]{algorithm2e}
\usepackage{svg}
\usepackage{longtable}
\usepackage{amsbsy}
\usepackage[noadjust]{cite}

\renewcommand{\QED}{\QEDopen}

\newcommand{\bmat}[1]{\begin{bmatrix}#1\end{bmatrix}}

\newcommand{\Prob}[1]{\mathbb{P} \left( #1 \right)}
\newcommand{\Exp}[1]{\mathbb{E} \left[ #1 \right]}
\newcommand{\Expx}[1]{\mathbb{E}_{x_0 \sim \mathcal{D}} \left[ #1 \right]}
\newcommand{\Expxt}[1]{\mathbb{E}_{x_0 \sim \mathcal{D}, \omega(0)\sim \rho} \left[ #1 \right]}

\newcommand{\1}[1]{\mathbf{1}_{#1}}

\newcommand{\R}{\mathbb{R}}

\newcommand{\tr}[1]{\text{tr} \left( #1 \right)}
\newcommand{\diag}[1]{\textup{diag}\! \left( #1 \right)}

\newcommand{\sumtinf}{\sum_{t=0}^\infty}

\newcommand{\EPnext}{\mathcal{E}_i(P)}
\newcommand{\EPKnext}{\mathcal{E}_i(P^{\hat{K}})}
\newcommand{\dEPKnext}{d\mathcal{E}_i(P^{\hat{K}})}
\newcommand{\Enext}[1]{\mathcal{E}_i(#1)}

\newcommand{\EPhatnext}{\hat{\mathcal{E}}(P^{\hat{K}})}

\newcommand{\deltaK}{\Delta \hat{K}}
\newcommand{\deltaKi}{\Delta K_i}

\newtheorem{mythm}{Theorem}
\newtheorem{mylemma}[mythm]{Lemma}
\newtheorem*{myassump}{Assumptions}

\newtheorem*{myprb}{Problem 1: Policy Optimization for MJLS}

\theoremstyle{definition}

\newtheorem*{myprf}{Proof Sketch}

\usepackage{hyperref}

\hypersetup{pdfauthor={Joao Paulo Jansch-Porto},%
            pdftitle={Convergence Guarantees of Policy Optimization Methods for Markovian Jump Linear Systems},%
            pdfsubject={MJLS; Switched Systems; Reinforcement Learning; Policy Gradient},%
            pdfkeywords={LQR; hybrid systems; MJLS; switched systems; Reinforcement Learning; Policy Gradient},%
            colorlinks  = true, 
            urlcolor    = blue, 
            linkcolor   = blue, 
            citecolor   = red, 
}

\title{\LARGE \bf
 Convergence Guarantees of Policy Optimization Methods for \\ Markovian Jump Linear Systems
}

\author{
Joao Paulo Jansch-Porto\thanks{J. P. Jansch-Porto is with the Department of Mechanical Science \& Engineering,
  University of Illinois at Urbana-Champaign, Email:   \texttt{janschp2@illinois.edu} }\and
  Bin Hu\thanks{B.~Hu is with the  Department of Electrical \& Computer Engineering,
  University of Illinois at Urbana-Champaign, Email:   \texttt{binhu7@illinois.edu}} \and
    Geir E. Dullerud \thanks{ G. E. Dullerud is with the Department of Mechanical Science \& Engineering,
  University of Illinois at Urbana-Champaign, Email:   \texttt{dullerud@illinois.edu} } 
}

\begin{document}

\maketitle

\begin{abstract}
Recently, policy optimization for control purposes has received renewed attention due to the increasing interest in reinforcement learning. 
In this paper, we investigate the convergence of policy optimization for quadratic control of Markovian jump linear systems (MJLS).
First, we study the optimization landscape of direct policy optimization for MJLS, and, in particular, show that despite the non-convexity of the resultant problem the unique stationary point is the global optimal solution. 
Next, we prove that the Gauss-Newton method and the natural policy gradient method converge to the optimal state feedback controller for MJLS at a linear rate if initialized at a controller which stabilizes the closed-loop dynamics in the mean square sense.
We propose a novel Lyapunov argument to fix a key stability issue in the convergence proof. 
Finally, we present a numerical example to support our theory.
Our work brings new insights for understanding the performance of policy learning methods on controlling unknown MJLS.
\end{abstract}

\section{Introduction}
\label{sec:intro}

Recently, reinforcement learning (RL) \cite{sutton2018reinforcement} has achieved impressive performance on a class of continuous control problems including locomotion~\cite{schulman2015high} and robot manipulation~\cite{levine2016end}. Policy-based optimization is the main engine behind these RL applications \cite{duan2016benchmarking}. Specifically, the natural policy gradient method \cite{kakade2002natural} and several related methods including TRPO \cite{schulman2015trust}, natural AC \cite{peters2008natural}, and PPO \cite{schulman2017proximal} are among the most popular RL algorithms for continuous control tasks. These methods enable flexible policy parameterizations, and are end-to-end in the sense that the control performance metrics are directly optimized. 

Despite the empirical successes of policy optimization methods, how to choose these algorithms for a specific control task is still more of an art than a science \cite{henderson2018deep, rajeswaran2017towards}.
This motivates a recent research trend focusing on understanding  the performances of RL algorithms on simplified benchmarks. Specifically, significant research has recently been conducted to understand the performance of various model-free or model-based RL algorithms on the classic Linear Quadratic Regulator (LQR) problem \cite{pmlr-v80-fazel18a, dean2017sample, malik2018derivative,tu2018least,tu2018gap, dean2018regret, abbasi2011regret, abbasi2018regret,Krauth2019,mohammadi2019convergence,mohammadi2019global}. 
In \cite{pmlr-v80-fazel18a}, it is shown that despite the non-convexity in the objective function, policy gradient methods can still provably learn the optimal LQR controller. This provides a good sanity check for policy optimization on further control applications. 

Built upon the good progress on understanding RL for the LQR problem, this paper moves one step further and studies policy optimization for Markov Jump Linear Systems (MJLS) \cite{costa2006discrete} from a theoretical perspective.
MJLS form an important class of systems that arise in many control applications \cite{bar1993estimation,fox2011tsp,hamsa2016cdc, Pavlovic2000LearningSL,sworder1999estimation,varga2013ijrnc}. Recently, stochastic methods in machine learning are also modeled as jump systems \cite{hu2017unified,hu2019characterizing}.  The research on MJLS has great practical value while in the mean time  also provides many new interesting theoretical problems. 
In the classic LQR problem, one aims to control a linear time-invariant (LTI) system whose state/input matrices do not change over time. On the other hand, the state/input matrices of a Markov jump linear system are functions of a jump parameter that is sampled from an underlying Markov chain. Consequently,
the behaviors of MJLS become very different from  those of LTI systems.  
Controlling unknown MJLS poses many new challenges over  traditional LQR due to the appearance of this Markov jump parameter. For example, in a model-based approach, one has to learn both the state/input matrices and the transition probability of the jump parameter; here, it is the coupling effect between the state/input matrices and the jump parameter distribution causes the main difficulty.
Therefore, the quadratic control of MJLS is a meaningful benchmark for further understanding of RL algorithms.

Obviously, studying policy optimization on MJLS control problems is important for further understanding of policy-based RL algorithms. 
In this paper, we present various convergence guarantees for policy optimization methods on the quadratic control of MJLS. First, we study the optimization landscape of direct policy optimization for MJLS, and demonstrate that despite the non-convexity of the resultant problem, the unique stationary point is the global optimal solution. Next, we prove that the Gauss-Newton method and the natural policy gradient method converge to the optimal state feedback controller for MJLS at a linear rate if a stabilizing initial controller is used. We introduce a novel Lyapunov argument to fix a key stability issue in the convergence proof. Finally, numerical simulations are provided to support our theory.

The most relevant reference of our paper is \cite{pmlr-v80-fazel18a}. Our results generalize the convergence theory of the Gauss-Newton method and the natural policy gradient method in \cite{pmlr-v80-fazel18a} to the MJLS case. This extension is non-trivial. Specifically,
one key issue in the convergence proof is to ensure that the iterates never wander into the region of instability. 
In \cite{pmlr-v80-fazel18a}, the system is LTI and the stability argument can be made by using the properties of spectral radius of the state matrix. For MJLS, one cannot directly make such arguments any more due to the stochastic nature of the system. Alternatively, we propose a novel Lyapunov argument to show that the resultant controller is always stabilizing for the MJLS in the mean square sense along the optimization trajectory of the Gauss-Newton method and the natural policy gradient method, if learning rates are chosen properly.

\section{Background and Preliminaries}
\label{sec:background}

\subsection{Notation}
We denote the set of real numbers by \(\R\).
Let \(Z\) be a square matrix, we use the notation \(Z^T\), \( \|Z\| \),  \(\tr{Z}\), \(\sigma_{\min}(Z)\) to denote its transpose, spectral norm, trace, and minimum singular value, respectively. We indicate positive definite and positive semidefinite matrices by \(Z\succ 0\) and \(Z \succeq 0\), respectively.
Given matrices \(\{D_i\}_{i = 1}^m\),  let \( \diag{D_1, \ldots, D_m}\)  denote the  block diagonal matrix whose $(i,i)$-th block is $D_i$. 
Given a function $f$, we use $df$ to denote its total derivative.

\subsection{Quadratic Control of Markovian Jump Linear Systems}
A Markovian jump linear system is governed by the following discrete-time state-space model
\begin{equation} \label{eq:ltv}
x_{t+1} = A_{\omega(t)} x_t + B_{\omega(t)} u_t
\end{equation}
where \( x_t \in \R^d\) is the system state at time \(t\in\mathbb{N}_0\), and \(u_t \in \R^k \) corresponds to the control action at time~\(t\).
The initial state \(x_0\) is assumed to have a distribution \(\mathcal{D}\).
 The system matrices \( A_{\omega(t)} \in \R^{d\times d}\) and \(B_{\omega(t)} \in \R^{d\times k} \) depend on the switching parameter $\omega(t)$, which takes values on \( \Omega:=\{1, \ldots, n_s \} \) for each $t$. Obviously, we have \(A_{\omega(t)}\in\{A_1, \ldots, A_{n_s}\}\) and \(B_{\omega(t)}\in\{B_1, \ldots, B_{n_s}\}\) for all $t$.
The jump parameter \(\omega(t)\) forms a discrete-time Markov chain sampled from \(\Omega \).
The transition probabilities and initial distribution of $\omega(t)$ are given by 
\begin{equation}\label{eq:prob}
p_{ij} = \Prob{\omega(t+1) = j | \omega(t) = i} \text{ and } \rho = \bmat{\rho_1\! &\! \!\cdots\!\! & \!\rho_{n_s}}^T
\end{equation}
respectively.
The transition probabilities satisfy \(p_{ij} \geq 0\) and \(\sum_{j=1}^{n_s} p_{ij} = 1\) for each \(i \in \Omega\). The initial distribution satisfies \( \sum_{i\in\Omega} \rho_i = 1 \).

In this paper, we focus on the quadratic control problem whose objective is to choose the control actions \(\{u_t\}\) to minimize the following cost function
\begin{equation} \label{eq:switched_cost}
C = \Expxt{\sumtinf x_t^T Q_{\omega(t)} x_t + u_t^T R_{\omega(t)} u_t},
\end{equation}
where it is assumed that \(Q_{\omega(t)} \succ 0\) and \(R_{\omega(t)} \succ 0\) for each \(t\).
This problem can be viewed as the MJLS counterpart of the standard LQR problem, and hence is termed as ``MJLS LQR problem."
The optimal controller to this MJLS LQR problem, defined by dynamics (\ref{eq:ltv}), cost (\ref{eq:switched_cost}), and switching probabilities (\ref{eq:prob}), can be computed by solving a system of coupled Riccati Algebraic Equations~\cite{fragoso}, which we now describe. First, it is known that 
the optimal cost can be achieved by a linear state feedback of the form
\begin{equation}\label{eq:control}
u_t = -K_{\omega(t)} x_t
\end{equation}
with \(  K_{\omega(t)} \in \R^{k \times d} \) for each \( \omega(t) \in \Omega \).
One can solve $K_i$ for all $i\in \Omega$ as follows.
Let $\mathcal{E}_i(P) \coloneqq \Exp{ P_{\omega(t+1)} \middle| \omega(t) = i } = \sum_{j = 1}^{n_s} p_{ij} P_j$.
Formally, let \( \{ P_i \}_{i\in\Omega} \) be the unique positive definite solution to the following equations:
\begin{align}\label{eq:markov_riccati}
P_i = Q_i &+  A_i^T \EPnext A_i - A_i^T \EPnext B_i \times \nonumber\\
&\left( R_i + B_i^T \EPnext B_i \right)^{-1} B_i^T \EPnext A_i.
\end{align}
It can be shown that the linear state feedback controller that minimizes the cost function (\ref{eq:switched_cost}) is given by
\begin{equation}\label{eq:opt_k}
K^*_i =\left( R_i + B_i^T \EPnext B_i \right)^{-1} B_i^T \EPnext A_i
\end{equation}

We remark that if \({\omega(t)} = {\omega(t+n_s)}\) for all \(t\), then the system is said to be periodic with period \( n_s \).
Linear periodic systems have been widely studied~\cite{Bittanti1991,293179} and are just a special case of MJLS.
If \(n_s = 1\), then the MJLS just becomes a linear time-invariant (LTI) system.

\subsection{Policy Optimization for Quadratic Control of LTI Systems}
\label{sec:LQRreview}
Before proceeding to policy optimization of MJLS, here we review policy gradient methods for the quadratic control of  LTI systems~\cite{pmlr-v80-fazel18a}. Consider the LTI system $x_{t+1}=A x_t+B u_t$ with an initial state distribution $\mathcal{D}$ and a static state feedback controller $u_t=-Kx_t$.
We adopt a standard quadratic cost function which can be calculated as
\begin{align} \label{eq:lti_cost}
C(K) &= \Expx{ \sum_{t=0}^\infty x_t^T Q x_t + u_t^T R u_t }\nonumber \\
&=\Expx{ \sum_{t=0}^\infty x_t^T (Q+K^T R K) x_t}.
\end{align}
Obviously, the cost in \eqref{eq:lti_cost} can be computed as
$C(K) = \Expx{x_0^T P_K x_0}$ where  $P_K$ is the solution to the Lyapunov equation
$P_K = Q + K^T R K + (A - B K)^T P_K (A - B K)$.
It is also well known~\cite{maartensson2009gradient, pmlr-v80-fazel18a} that the gradient of~(\ref{eq:lti_cost}) with respect to \(K\)  can be calculated as
\begin{equation*}
\nabla C(K) = 2 \left(  \left( R + B^T P_K B\right) K - B^T P_K A \right) \Sigma_K.
\end{equation*}
where $\Sigma_K$ is the state correlation matrix, i.e. 
$\Sigma_K  = \Expx{\sum_{t=0}^\infty x_t x_t^T}$. 
Based on this gradient formula, one can optimize \eqref{eq:lti_cost} using the (deterministic) policy gradient method $K'\leftarrow K-\eta \nabla C(K)$, the natural policy gradient method $K'\leftarrow K-\eta \nabla C(K) \Sigma_K^{-1}$, or the Gauss-Newton method $K'\leftarrow K-\eta (R+B^TPB)^{-1} \nabla C(K)\Sigma_K^{-1}$. More explanations for these  methods can be found in \cite{pmlr-v80-fazel18a}.

In \cite{pmlr-v80-fazel18a}, it is shown that there exists a unique $K^*$ such that $\nabla C(K^*)=0$ if $\Expx{ x_0 x_0^T}$ is full rank.  In addition, all the above methods are shown to converge to $K^*$ linearly if a stabilizing initial policy is used.

\section{Policy Gradient and Optimization Landscape}
\label{sec:pg}
Now we focus on the policy optimization of the MJLS LQR problem. Since we know the optimal cost can be achieved by a linear state feedback, it is reasonable to  restrict the policy search within the class of linear state feedback controllers.
Specifically, we can set \(\hat{K} = \bmat{K_1 & \cdots & K_{n_s} }\), where each of the components is the feedback gain of the corresponding mode. 
With this notation, we 
consider the following policy optimization problem whose  decision variable is $\hat{K}$.

\begin{myprb}
\begin{align*}
\text{minimize:} &\quad \text{cost } C(\hat K),\text{ given in (\ref{eq:switched_cost})}\\
\null \text{subject to:} &\quad \text{state dynamics, given in (\ref{eq:ltv})} \\
        &\quad \text{control actions, given in (\ref{eq:control})}\\
        &\quad \text{transition probabilities, given in (\ref{eq:prob})}\\
   & \quad \text{stability constraint, } \hat{K}\text{ stabilizing \eqref{eq:ltv} in the}\\
   & \qquad \text{mean square sense.}
\end{align*}
\end{myprb}

In this section, we present an explicit formula for the policy gradient $\nabla C(\hat{K})$ and discuss the optimization landscape for the above problem.
We want to emphasize that the above problem is indeed a constrained optimization problem. The feasible set of the above problem consists of all $\hat{K}$ stabilizing the closed-loop dynamics in the mean square sense (and hence yielding finite $C(\hat K)$). We denote this feasible set as $\mathcal{K}$.
For $\hat K \notin \mathcal{K}$, the cost in \eqref{eq:switched_cost} can blow up to infinity, and the differentiability is also an issue.
For $\hat K \in \mathcal{K}$, the cost $C(\hat K)$ is finite and differentiable.
To obtain the formula for $\nabla C(\hat K)$, we can first  rewrite the quadratic cost (\ref{eq:switched_cost}) as
\begin{align} \label{eq:markov_cost}
C(\hat K) &=  \Expxt{x_0^T P_{\omega(0)}^{\hat K} x_0}\nonumber \\
&= \Expx{x_0^T \left( \sum_{i\in\Omega } \rho_i P_i^{\hat{K}} \right) x_0},
\end{align}
where $P_i^{\hat{K}}$ is defined to be the solution to the coupled Lyapunov equations
\begin{equation}\label{eq:lyap_markov}
P^{\hat{K}}_i = Q_i + K_i^T R_i K_i + \left( A_i - B_i K_i \right)^T \EPKnext \left( A_i - B_i K_i \right)
\end{equation}
for all $i \in \Omega$.
Recall that we have $\EPKnext = \sum_{j = 1}^{n_s} p_{ij} P_j^{\hat K}$.

We will denote \(X_i(t) \coloneqq \Exp{x_t x_t^T \1{ \omega(t) = i}}\).
This matrix also satisfies the recurrence \cite{costa2006discrete}:
\begin{equation*}
X_j(t+1) = \sum_{i \in \Omega} p_{ij} (A_i - B_i K_i) X_i(t) (A_i - B_i K_i)^T
\end{equation*}
with \(X_i(0) = \rho_i \Expx{x_0 x_0^T}\) for all \(i \in \Omega\).

We also make the following technical assumptions.
\begin{myassump}
Along with the standard assumption that $Q_i$ and $R_i$ being positive definite for all $i \in \Omega$, we assume that $\rho_i > 0$ for all $i \in \Omega$ and  \(\Expx{x_0 x_0^T} \succ 0\).
This indicates that there is a chance of starting from any mode $i$. Moreover, the expected covariance of the initial state is full rank.
\end{myassump}

Now we are ready to present an explicit formula for the policy gradient \(\nabla C(\hat{K})\).

\begin{mylemma} \label{lemma:policy_grad}
Given $\hat{K}\in \mathcal{K}$, the gradient for the cost function~(\ref{eq:markov_cost}) with respect to policy \(\hat{K}\) is
\begin{equation}\label{eq:exact_grad}
\nabla C(\hat{K}) = 2\bmat{L_1(\hat{K}) & L_2(\hat{K}) & \cdots & L_{n_s}(\hat{K})} \chi_{\hat{K}}
\end{equation}
where 
\begin{equation}
\label{eq:Ldef}
L_i(\hat{K}) = \left(R_i + B^T_i \EPKnext B_i \right) K_i - B^T_i \EPKnext A_i
\end{equation}
and
\begin{equation}
\label{eq:chidef}
\chi_{\hat{K}} =  \textup{diag} \left( \sumtinf X_1(t), \ldots, \sumtinf X_{n_s}(t) \right).
\end{equation}
\end{mylemma}

\begin{proof}
The differentiability of $ C(\hat K)$  can be proved using the implicit function theorem, and this step is similar to the proof of Lemma 3.1 in \cite{rautert1997computational}. Now we derive the gradient formula by modifying the total derivative arguments in \cite{rautert1997computational, maartensson2009gradient}.
Start by denoting \( \phi_i = A_i - B_i K_i \). Then, we can take the total derivative of (\ref{eq:lyap_markov}) to show the following relation for each $i \in \Omega$ 
\begin{align*}
dP_i^{\hat{K}} &= dK_i^T L_i(\hat{K}) + L_i(\hat{K})^T dK_i + \phi_i^T \! \left( \sum_{j\in\Omega} p_{ij} dP_j^{\hat{K}}\! \right) \phi_i\\
	&= dK_i^T L_i(\hat{K}) + L_i(\hat{K})^T dK_i + \phi_i^T \dEPKnext \phi_i
\end{align*}
Hence, the total derivative of the cost \eqref{eq:markov_cost} is
\begin{align*}
dC(\hat{K}) &=\Expx{ x_0^T \left( \sum_{i\in\Omega } \rho_i \, dP_i^{\hat{K}} \right) x_0} \\
&= \tr{\left( \sum_{i\in\Omega } \rho_i \, dP_i^{\hat{K}} \right) \Expx{ x_0 x_0^T}}\\
	&= \text{tr}\left(2 \sum_{i\in\Omega} \left( \rho_i \left(dK_i^T L_i(\hat{K})\right) \Expx{x_0 x_0^T} \right)\right. \\
	&\qquad\quad \left. + \sum_{i\in\Omega} \left( \rho_i \dEPKnext \phi_i \Expx{x_0 x_0^T} \phi_i^T \right)\right)\\
	&= \tr{2 \sum_{i\in\Omega}  dK_i^T L_i(\hat{K})  \left( \sumtinf X_i(t) \right) }.
\end{align*}
Recall from \cite{rautert1997computational} that $dC(K) = \tr{d\hat{K}^T\nabla C(\hat{K}) }$. This leads to the desired result.
\end{proof}

\textbf{Optimization Landscape for MJLS.}
LTI systems are just a special case of MJLS.
Since  policy optimization for quadratic control of LTI systems is non-convex, the same is true for the MJLS case.
However, from our gradient formula in Lemma \ref{lemma:policy_grad}, we can see that as long as \(\Expx{x_0 x_0^T}\) is 
full rank and $\rho_i>0$ for all $i$, it is necessary that a stationary point given by $\nabla C(\hat{K})=0$   satisfies
\begin{equation*}
L_i(\hat{K}) = \left(R_i + B^T_i \EPKnext B_i \right) K_i - B^T_i \EPKnext A_i=0.
\end{equation*}
Substituting the above equation into the coupled Lyapunov equation \eqref{eq:lyap_markov} leads to the global solution $\hat{K}^*$ defined by the coupled Algebraic Riccati Equations~(\ref{eq:markov_riccati}).
Therefore, the only stationary point is the global optimal solution. Overall, the optimization landscape for the MJLS case is quite similar to the classic LQR case if we allow the initial mode to be sufficiently random, i.e. $\rho_i>0$ for all $i$. Based on such similarity, it is reasonable to expect that the local search procedures (e.g. policy gradient) will be able to find the unique global minimum $\hat{K}^*$  for MJLS despite the non-convex nature of the problem. Compared with the LTI case, the characterization of $\mathcal{K}$ is more complicated for MJLS. Hence the main technical issue is how to show gradient-based methods can handle the feasibility constraint $\hat{K}\in \mathcal{K}$ without using projection. We will use a Lyapunov argument to tackle this issue.

\section{Main Convergence Results}

As reviewed in Section \ref{sec:LQRreview}, the  natural policy gradient method for the LTI case iterates as $K'\leftarrow K-\eta \nabla C(K) \Sigma_K^{-1}$.
For the MJLS case, the natural policy gradient method adopts a similar update rule and iterates as
\begin{align}
\label{eq:npgd}
\hat{K}^{n+1}= \hat{K}^n-\eta \nabla C(\hat{K}^n) \chi_{\hat{K}^n}^{-1}.
\end{align}
The  initial policy is denoted as $\hat{K}^0$.
The Gauss-Newton method uses the following update rule:
\begin{equation}
\label{eq:gngd}
\hat{K}^{n+1}= \hat{K}^n-2\eta \bmat{\psi_1^n L_1(\hat{K}^n), \ldots,\psi_{n_s}^n L_{n_s}(\hat{K}^n)} 
\end{equation}
where  $\psi_i^n \coloneqq (R_i + B_i^T \mathcal{E}_i(P^{\hat{K}^n}) B_i)^{-1}$.
In this section, we focus on the convergence guarantees of \eqref{eq:npgd} and \eqref{eq:gngd}, and show that both converge to the global optimal solution $\hat{K}^*$ at a linear rate if they are initialized at a policy in $\mathcal{K}$. 

To state the main convergence result, it is helpful to denote \(\hat{R} = \diag{R_1, \ldots, R_{n_s}}\), \(\hat{B} = \diag{B_1 , \ldots , B_{n_s}}\), and \(\EPhatnext=\diag{\mathcal{E}_1(P^{\hat{K}}),\ldots,\mathcal{E}_{n_s}(P^{\hat{K}})}\).
We also denote $\mu \coloneqq \min_{i\in \Omega}(\rho_i)\, \sigma_{\min}\!\left( \Expx{x_0 x_0^T} \right)$.

\begin{mythm}\label{thm:conv}
Suppose \(\hat{K}^0\in \mathcal{K}\) s.t. \( C(\hat{K}^0) \) is finite.
\begin{itemize}
    \item Gauss-Newton case: For any stepsize $\eta \le \frac{1}{2}$,
the Gauss-Newton method (\ref{eq:gngd}) converges to the global minimum $\hat{K}^*\in \mathcal{K}$ linearly as follows
\begin{align}
\label{eq:mainConGN}
C(\hat{K}^n) - C(\hat{K}^*) &\le \left( 1 -  \frac{2\eta\mu}{\| \chi_{\hat{K}^*} \|} \right)^n\times \nonumber \\
& \qquad \quad \left(C(\hat{K}^0) - C(\hat{K}^*) \right).
\end{align}
    \item Natural policy gradient case: For any stepsize $\eta \le \frac{1}{2\left( \| \hat{R} \| + \frac{\| \hat{B} \|^2 C(\hat{K}^0)}{\mu} \right)}$,
the  natural policy gradient method (\ref{eq:npgd}) converges to the global minimum $\hat{K}^*\in \mathcal{K}$ linearly as follows
\begin{align}
\label{eq:mainCon}
C(\hat{K}^n) - C(\hat{K}^*) &\le \left( 1 - \frac{2\eta\mu  \sigma_{\min}(\hat{R})}{\| \chi_{\hat{K}^*} \|} \right)^n\times \nonumber \\
& \qquad \quad \left(C(\hat{K}^0) - C(\hat{K}^*) \right).
\end{align}
\end{itemize}

\end{mythm}
\begin{myprf}
We briefly outline the main  proof  steps for the Gauss-Newton case. The proof for the natural policy gradient case is similar. The detailed proofs are presented in the appendix.
\begin{enumerate}
    \item Show that the one-step progress of the Gauss-Newton gives a policy stabilizing the closed-loop dynamics and yielding a finite cost.
    \item Apply the so-called ``almost smoothness" condition to show that the cost associated by the one-step progress of the Gauss-Newton  method decreases as follows:
\begin{align*}
C(\hat{K}^{n+1}) - C(\hat{K}^*) &\le \left( 1 -  \frac{2\eta\mu}{\| \chi_{\hat{K}^*} \|} \right) \times \\
& \qquad \qquad \left(C(\hat{K}^n) - C(\hat{K}^*) \right).
\end{align*}
    \item Use induction to show the final convergence result.
\end{enumerate}
 It is worth noting that the proof steps for the MJLS and LTI cases are quite similar. We can simply modify the proof arguments for the LTI case in \cite{pmlr-v80-fazel18a} to finish the second and third steps. The main challenge for the MJLS case is how to handle the first step, since one cannot directly modify the spectral radius argument in \cite{pmlr-v80-fazel18a} due to the stochastic nature of MJLS. We develop a novel Lyapunov argument to address this issue. We will only present the details for the first step here since that is the only part requiring new proving techniques. Other steps of the proof for both cases are deferred to the appendix.
\end{myprf}

\textbf{How do the policy optimization methods ensure the finite cost along the iteration path?} We need to show that for every \(\hat{K}\), we can choose a step size \(\eta\) such that the new controller  obtained in one-step update of the Gauss-Newton method or the natural policy gradient method (which is denoted as $\hat{K}'$) will also be stabilizing the closed-loop dynamics in the mean-square sense.
This lemma is of new technical novelty compared with the argument for the LTI case in \cite{pmlr-v80-fazel18a}. 
Notice that the ``almost smoothness" condition is required in the second step of the proof outline, as it gives a useful upper bound for $C(\hat{K}')$ in terms of $C(\hat{K})$.
However, to apply such a condition, one needs to ensure that both \(\hat{K}'\) and \(\hat{K}\) are stabilizing controllers such that $C(\hat{K}')$ and $C(\hat K)$ are both finite in the first place. 
Hence, one has to prove that the iterate $\hat{K}'$ never wanders into the region of instability before applying the ``almost smoothness" condition.

To show that every controller computed by the Gauss-Newton method or the natural policy gradient method  is stabilizing, we propose the following Lyapunov argument.
The main idea is that the value function at the current step serves naturally as a Lyapunov function for the next move due to the positive definiteness of \(Q_i\). 
The positive definiteness of \(Q_i\) guarantees that there is a stability margin around every point along the optimization trajectory.
The result for the Gauss-Newton case is formally stated below.

\begin{mylemma}\label{lemma:step_stable_GN} 
Suppose \(\hat{K}\) stabilizes the MJLS \eqref{eq:ltv} in the mean square sense.  
Then the one-step update $\hat{K}'$ obtained from the Gauss-Newton method~\eqref{eq:gngd} will also be stabilizing if the step size \(\eta\) satisfies $\eta \leq \frac{1}{2}$.
\end{mylemma}
\begin{proof}
Recall from \cite{costa1993stability} that the controller $\hat{K}'$ stabilizes \eqref{eq:ltv} in the mean-square sense if and only if there exists matrices \( \{Y_i\} \succ 0 \) such that
\begin{equation}
\label{eq:LMI}
(A_i - B_i K_i')^T \left[ \sum_{j \in \Omega} p_{ij} Y_j \right] (A_i - B_i K_i') - Y_i \prec 0, \quad \forall i \in \Omega
\end{equation}
We will show that the above condition can be satisfied by setting $Y_i=P_i$ where $P_i$ solves the MJLS Lyapunov equation $(A_i - B_i K_i)^T \EPKnext (A_i - B_i K_i) + Q_i + K_i^T R_i K_i = P_i$. Notice the existence of $P_i$ is guaranteed by the assumption $\hat{K}\in \mathcal{K}$.
Denote $\deltaK_i:= K_i-K_i'$. The Lyapunov equation for $P_i$ can be rewritten as $(A_i\! -\! B_i K_i'\! -\! B_i \deltaKi )^T \EPKnext (A_i\! -\! B_i K_i'\! -\! B_i \deltaKi)+ Q_i + (K_i' + \deltaKi)^T R_i (K_i' + \deltaKi) = P_i $. From this, we can directly obtain 
\begin{align*}
& (A_i - B_i K_i')^T \EPKnext (A_i - B_i K_i') - P_i = \\
&\qquad \qquad -\left( Q_i + (K_i')^T R_i K_i'\right) \\
&\qquad \qquad - \left(\deltaKi^T R_i \deltaKi + \deltaKi^TB_i^T \EPKnext B_i \deltaKi \right) \\
&\qquad \qquad - \deltaKi^T \left( R_i K_i' - B_i^T \EPKnext (A_i - B_i K_i') \right) \\
&\qquad \qquad - \left( R_i K_i' - B_i^T \EPKnext (A_i - B_i K_i') \right)^T \deltaKi
\end{align*}
Since \(R_i, \EPKnext, Q_i\) are all positive definite, the sum of the first two terms on the right hand side is negative definite.
We only need the last two terms to be negative semidefinite.
Note that $\deltaKi = 2 \eta (R_i + B_i^T \EPKnext B_i)^{-1} L_i(\hat{K})$. We have
\begin{align*}
&\deltaKi^T \left( R_i K_i' - B_i^T \EPKnext (A_i - B_i K_i') \right) \\
&\quad = \deltaKi^T\! \left( \left(\!R_i +  B_i^T \EPKnext B_i\right)\! K_i' - B_i^T \EPKnext A_i \right) \\
	&\quad= \deltaKi^T\! \left( - \left(\!R_i +  B_i^T \EPKnext B_i\right)\! \Delta K_i + L_i(\hat{K}) \right) \\
	&\quad= 2 \eta(1 - 2\eta) L_i(\hat{K})^T\left(\!R_i +  B_i^T \EPKnext B_i\right)^{-1}  L_i(\hat{K})
\end{align*}
which is positive semidefinite under the condition $\eta\le \frac{1}{2}$.
\end{proof}

For the natural gradient method, we have a similar result.
\begin{mylemma}\label{lemma:step_stable} 
Suppose \(\hat{K}\) stabilizes the MJLS \eqref{eq:ltv} in the mean square sense.  
Then the one-step update $\hat{K}'$ obtained from the natural policy gradient method~\eqref{eq:npgd} will also be stabilizing if  \(\eta\) satisfies 
\begin{equation*}
 \eta \leq \frac{1}{2 \| \hat{R} + \hat{B}^T \EPhatnext \hat{B} \|}.
\end{equation*}
\end{mylemma}
\begin{proof}
The proof starts with the same steps as the proof of Lemma~\ref{lemma:step_stable_GN}. We will show that the condition \eqref{eq:LMI} can be met by setting $Y_i=P_i$ where $P_i$ solves the MJLS Lyapunov equation associated with the controller $\hat{K}$.
For the natural policy gradient method, we have \(\deltaK \coloneqq \eta \nabla C(\hat{K}) \chi_{\hat{K}}^{-1}\) and \(\deltaKi = 2\eta L_i(\hat{K}) \).
To show that the last two terms are negative semidefinite, we make the following calculations:
\begin{align*}
&\deltaKi^T \left( R_i K_i' - B_i^T \EPKnext (A_i - B_i K_i') \right) \\
=& \deltaKi^T\! \left( (\!R_i +  B_i^T \EPKnext B_i)\! K_i' - B_i^T \EPKnext A_i \right) \\
= &\deltaKi^T\! \left( (\!R_i +  B_i^T \EPKnext B_i)\! \left(K_i - \deltaKi \right)  - B_i^T \EPKnext A_i \right) \\
= &2\eta L_i(\hat{K})\! \left(\! L_i(\hat{K}) -\! 2\eta\!  \left(\!R_i +  B_i^T \EPKnext B_i\!\right)\! L_i(\hat{K})\! \right) \\
=& 2\eta L_i(\hat{K})\! \left(I - 2\eta \left(R_i +  B_i^T \EPKnext B_i\right) \right)  L_i(\hat{K})
\end{align*}
Clearly, the above term is guaranteed to be positive semidefinite if $\eta$ satisfies
\begin{align*}
 \eta \leq \frac{1}{2\| R_i + B_i^T \EPKnext B_i \|}.
\end{align*}
Lastly, notice
$\| R_i + B_i^T \EPKnext B_i \| \leq \| \hat{R} + \hat{B}^T \EPhatnext \hat{B} \|$ for all $i$.
This leads to the desired conclusion.
\end{proof}
From the above proof, we can clearly see that $P^{\hat K}$ can be used to construct a Lyapunov function for $K'$ if $\eta$ satisfies a bound. This leads to a novel proof for the stability along the natural policy gradient iteration path. This idea can even be extended beyond the linear quadratic control case. Very recently, a similar idea is used to show the convergence properties of policy optimization methods for the  mixed $\mathcal{H}_2/\mathcal{H}_\infty$ control problem where the cost function may not blow up to infinity on the boundary of the feasible set \cite{zhang2019policyc}.

\section{Simulation Results}
Consider a system with 100 states, 20 inputs, and 100 modes.
The system matrices \(A\) and \(B\) were generated using \texttt{drss} in MATLAB in order to guarantee that the system would have finite cost with \(\hat{K}^0 = 0\).
The probability transition matrix \(\mathcal{P}\) was sampled from a Dirichlet Process \( \text{Dir}(99\cdot I_{100} + 1) \).
We also assumed that we had equal probability of starting in any initial mode.
For simplicity we set \( Q_i = I \) and \( R_i = I \) for all \( i \in \Omega \).

\begin{figure}[h!]
\centering
\includegraphics[width=0.485\textwidth]{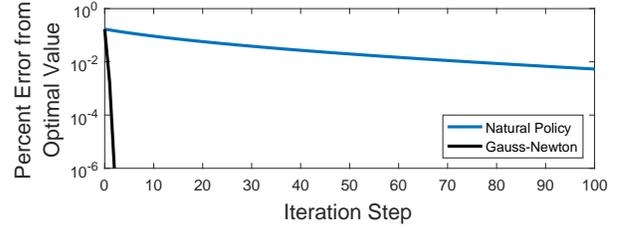}
\caption{Percent error from optimal cost for controllers computed using both policy optimization methods. The Gauss-Newton method converges faster.}
\label{fig:exact_conv}
\end{figure}

In Figure~\ref{fig:exact_conv} we can see that both policy optimization methods converge to the optimal solution.
As expected, Gauss-Newton converges much faster than the natural policy gradient method.
The step size of the natural policy gradient method depends on various system parameters, and requires some tuning efforts for each different problem instance.

\section{Conclusion}
In this paper we have studied policy optimization for the quadratic control of Markovian jump linear systems.
We developed an exact formula for computing the gradient of the cost with respect to a given policy, and presented convergence guarantees for the Gauss-Newton method and the natural policy gradient method.
The results include a novel Lyapunov argument to prove the stability of the iterations along the optimization trajectories.

The results obtained further suggest that one could use model-free methods, such as zeroth-order optimization or the REINFORCE algorithm, to learn the optimal control from data.
Such model-free techniques will allow us  to learn the control of unknown MJLS without dealing with system identification. 
This would be particularly useful for large scale systems, where the computational complexity grows as the system size increases.
We will work on such extensions in the future.

\bibliographystyle{IEEEtran}
\bibliography{ACC2020}

\section*{APPENDIX}
This appendix provides the detailed proofs of the convergence rate results presented in this paper.
We will first start by proving a few helper lemmas. Then we will provide upper bounds for the cost associated with the one-step progress.
Lastly we will show that the both algorithms converge to the optimal policy. Most steps mimic their LTI counterparts.

\begin{mylemma}[``Almost smoothness'']
\label{lemma:smooth}
Suppose $\hat{K}\in \mathcal{K}$ and $\hat{K}'\in \mathcal{K}$.
The cost function \(C(\hat{K})\) defined in (\ref{eq:switched_cost}) satisfies
\begin{align*}
C(\hat{K}') - C(\hat{K}) &= -2\tr{\chi_{\hat{K}'} \Delta \hat{K}^T \hat{L}_{\hat{K}}} \\
&\quad+ \tr{\chi_{\hat{K}'} \Delta \hat{K}^T\! \left(\hat{R}\! +\! \hat{B}^T \EPhatnext \hat{B}\right)\!\Delta \hat{K}}
\end{align*}
\begin{align*}
\text{where }\Delta \hat{K} &= \diag{(K_1 - K_1'), \ldots, (K_{n_s} - K_{n_s}')},\\
\hat{L}_{\hat{K}} &= \diag{L_1(\hat{K}), \ldots, L_{n_s}(\hat{K})}, \\
\EPhatnext &= \diag{\mathcal{E}_1(P^{\hat{K}}), \ldots, \mathcal{E}_{n_s}(P^{\hat{K}})}.
\end{align*}
\end{mylemma}
\begin{proof}
To simplify the equations, we use \( \phi_i = A_i - B_i K_i \) and \( \phi_i' \coloneqq A_i - B_i K_i' \). 
By definition, we have
\begin{align}
\label{eq:Cdiff}
\begin{split}
C(\hat{K}') - C(\hat{K})&= \sum_{i\in\Omega}  \Expx{x_0\left( \rho_i(P^{\hat{K}'}_i - P^{\hat{K}}_i)\right)x_0} \\
	&=\sum_{i\in\Omega} \Expx{\tr{\rho_i(P^{\hat{K}'}_i - P^{\hat{K}}_i) x_0 x_0^T}} 
\end{split}
\end{align}
Now we develop a formula for $(P_i^{\hat{K}'}\!\! - P_i^{\hat{K}})$.
Based on (\ref{eq:lyap_markov}), we have $
P^{\hat{K}'}_i = (\phi_i')^T \Enext{P^{\hat{K}'}} \phi_i' + Q_i + (K_i')^T R_i K_i'$.  Using this, we can directly show
\begin{align*}
P_i^{\hat{K}'}\!\! - P_i^{\hat{K}} &= (\phi_i')^T \Enext{P^{\hat{K}'}} \phi_i' + Q_i + (K_i')^T R_i K_i' - P_i^{\hat{K}} \\
 &= (\phi_i')^T \left( \Enext{P^{\hat{K}'}} - \EPKnext \right) (\phi_i')+ Q_i \\
&  \quad + (\phi_i')^T \EPKnext \phi_i' + (K_i')^T R_i K_i' - P_i^{\hat{K}}\\
 &= (\phi_i')^T \left( \Enext{P^{\hat{K}'}} - \EPKnext \right) \phi_i' \\
&  \quad + (K_i - K_i') (R_i + B_i^T \EPKnext B_i)(K_i - K_i')\\
&  \quad - (K_i - K_i')^T\left(R_i K_i - B_i^T \EPKnext \phi_i\right) \\
&  \quad - \left(R_i K_i - B_i^T \EPKnext \phi_i\right)^T(K_i - K_i') 
\end{align*}
Now we can substitute the above formula into \eqref{eq:Cdiff} and  iterate to get the desired result.
\end{proof}

Next, we show that \(C(\hat{K})\) is gradient dominated. 
Recall that we have $\mu = \min_{i\in \Omega}(\rho_i)\, \sigma_{\min}\!\left( \Expx{x_0 x_0^T} \right)$.
\begin{mylemma}
[Gradient Domination]
\label{lemma:grad_dom}
Suppose  \(\hat{K}\in\mathcal{K}\), and \(\mu > 0\). Let \(\hat{K}^*\) be the optimal policy.
Given the definitions in Lemma~\ref{lemma:policy_grad},  the following sequence of inequalities always holds:
\begin{align*}
C(\hat{K})\! - C(\hat{K}^*)\! &\leq  \| \chi_{\hat{K}^*} \| \tr{\hat{L}_{\hat{K}}^T (\hat{R} + \hat{B}^T\EPhatnext \hat{B})^{-1} \hat{L}_{\hat{K}}}\\
	&\leq\! \frac{\| \chi_{\hat{K}^*} \|}{\sigma_{\min}(\hat{R})} \tr{\hat{L}_{\hat{K}}^T \hat{L}_{\hat{K}}} \\
	&\leq\! \frac{\| \chi_{\hat{K}^*} \|}{\mu^2 \sigma_{\min}(\hat{R})} \tr{\nabla C(\hat{K})^T \nabla C(\hat{K})},
\end{align*}
\end{mylemma}
\begin{proof}
For readability, we denote $\hat{\Psi} \coloneqq \hat{R} + \hat{B}^T \EPhatnext \hat{B} $. From Lemma~\ref{lemma:smooth}, we can complete the squares to show
\begin{align*}
&2\tr{\chi_{\hat{K}'} (-\Delta \hat{K})^T \hat{L}_{\hat{K}}} + \tr{\chi_{K'} (-\Delta \hat{K})^T \hat{\Psi}(-\Delta \hat{K})} \\
&\   = \tr{\!\chi_{\hat{K}'} \left(\!(-\Delta \hat{K}) +\hat{\Psi}^{-1} \hat{L}_{\hat{K}}\right)^T\!\! \hat{\Psi} \left(\!(-\Delta \hat{K}) +\hat{\Psi}^{-1} \hat{L}_{\hat{K}}\right)\!\!}\\
&\  \qquad - \tr{\chi_{\hat{K}'} \hat{L}_{\hat{K}}^T (\hat{\Psi})^{-1} \hat{L}_{\hat{K}}}\\
&\  \geq  -\tr{\chi_{\hat{K}'} \hat{L}_{\hat{K}}^T (\hat{R} + \hat{B}^T\EPhatnext \hat{B})^{-1} \hat{L}_{\hat{K}}}
\end{align*}

Then, from Lemma~\ref{lemma:smooth} we have the following inequality
\begin{align*}
C(\hat{K})\! - C(\hat{K}^*)\! &\leq \tr{\chi_{K^*} \hat{L}_{\hat{K}}^T (\hat{R} + \hat{B}^T\EPhatnext \hat{B})^{-1} \hat{L}_K}\\
	&\leq \| \chi_{\hat{K}^*} \| \tr{\hat{L}_{\hat{K}}^T (\hat{R} + \hat{B}^T\EPhatnext \hat{B})^{-1} \hat{L}_{\hat{K}}}\\
&\leq \frac{\| \chi_{\hat{K}^*} \|}{\sigma_{\min}(\hat{R})} \tr{\hat{L}_{\hat{K}}^T \hat{L}_{\hat{K}}}\\
&= \frac{\| \chi_{\hat{K}^*} \|}{\sigma_{\min}(\hat{R})} \tr{\chi_{\hat{K}}^{-1} \nabla C(\hat{K})^T \nabla C(\hat{K}) \chi_{\hat{K}}^{-1}} \\
&\leq\! \frac{\| \chi_{\hat{K}^*} \|}{\sigma_{\min}(\hat{R}) \sigma_{\min}(\chi_{\hat{K}})^2} \tr{\!\nabla C(\hat{K})^T\! \nabla C(\hat{K})\!} \\
&\leq \frac{\| \chi_{\hat{K}^*} \|}{\mu^2 \sigma_{\min}(\hat{R})} \tr{\nabla C(\hat{K})^T \nabla C(\hat{K})}
\end{align*}
This completes the proof.
\end{proof}

The next lemma gives a useful lower bound on the cost.  
\begin{mylemma} \label{lemma:bound}
Given the definitions in  (\ref{eq:lyap_markov}), the following holds
\begin{equation*}
\sum_{i\in \Omega}\| P_i^{\hat{K}} \| \leq \frac{C(\hat{K})}{\mu}.
\end{equation*}
\end{mylemma}
\begin{proof}
Recall  $C(\hat{K}) = \Expx{\tr{\left( \sum_{i\in \Omega} \rho_i P_i^{\hat{K}} \right) x_0 x_0^T}}$. Therefore, we have
\begin{align*}
C(\hat{K}) 	&\geq \tr{ \sum_{i\in \Omega} \rho_i P_i^{\hat{K}} } \sigma_{\min}\left({\Expx{x_0x_0^T}}\right)\\
	&\geq  \left( \sum_{i\in \Omega} \| P_i^{\hat{K}} \| \right)  \min_{i\in\Omega}(\rho_i) \sigma_{\min}\left({\Expx{x_0x_0^T}}\right)
\end{align*}
which gives the desired lower bound.
\end{proof}

The next lemma bounds the cost for the one-step progress of the Gauss-Newton method. 
\begin{mylemma}\label{lemma:one_stepGN}
If $\hat{K}'= \hat{K}-2\eta \bmat{\psi_1 L_1(\hat{K}), \ldots,\psi_{n_s} L_{n_s}(\hat{K})}$
with  $\psi_i \coloneqq (R_i + B_i^T \mathcal{E}_i(P^{\hat{K}}) B_i)^{-1}$
and \(\eta \leq \frac{1}{2}\),
then the following inequality holds for any $\hat{K}\in \mathcal{K}$:
\begin{align*}
C(\hat{K}') - C(\hat{K}^*) &\leq \left( 1 -  \frac{2\eta\mu}{\| \chi_{\hat{K}^*} \|} \right)\left( C(\hat{K}) - C(\hat{K}^*) \right).
\end{align*}
\end{mylemma}
\begin{proof}
First we can show that \(\hat{K}'\) is a stabilizing policy by applying Lemma~\ref{lemma:step_stable_GN}.
From Lemma \ref{lemma:smooth}, we have
\begin{align*}
C(\hat{K}')\! - C(\hat{K})\! &= -4\eta \tr{\!\chi_{\hat{K}'\!} \hat{L}_{\hat{K}}^T \hat{\Psi}^{-1} \hat{L}_{\hat{K}}}\! \\
&\qquad \qquad + 4 \eta^2 \tr{\!\chi_{\hat{K}'} \hat{L}_{\hat{K}}^T \hat{\Psi}^{-1} \hat{L}_{\hat{K}}\!} \\
	&\leq -2\eta \tr{\chi_{\hat{K}'} \hat{L}_{\hat{K}}^T \hat{\Psi}^{-1} \hat{L}_{\hat{K}}} \\
	&\leq -2\eta \sigma_{\min}(\chi_{\hat{K}'}) \tr{\hat{L}_K^T \hat{\Psi}^{-1} \hat{L}_{\hat{K}}} \\
	&\leq -2\eta \mu \tr{\hat{L}_{\hat{K}}^T (\hat{R} + \hat{B}^T\EPhatnext \hat{B})^{-1} \hat{L}_{\hat{K}}} \\
	&\leq -2\eta \frac{\mu}{\| \chi_{\hat{K}^*} \|} \left(C(\hat{K}) - C(\hat{K}^*)\right)
\end{align*}
where the last step follows from Lemma~\ref{lemma:grad_dom}.
\end{proof}

The next lemma bounds the cost of the one-step progress of the natural policy gradient method.
\begin{mylemma}\label{lemma:one_step}
Suppose $\hat{K}\in \mathcal{K}$.
If $\hat{K}' = \hat{K} - \eta \nabla C(\hat{K}) \chi_{\hat{K}}^{-1}$ and   
\begin{equation*}
 \eta \leq \frac{1}{2\| \hat{R} + \hat{B}^T \EPhatnext \hat{B} \|},
\end{equation*}
then the following inequality holds
\begin{align*}
C(\hat{K}') - C(\hat{K}^*) &\leq \left( 1 - \frac{ 2\eta\mu \sigma_{\min}(\hat{R})}{\| \chi_{\hat{K}^*} \|} \right)\left( C(\hat{K}) - C(\hat{K}^*) \right).
\end{align*}
\end{mylemma}
\begin{proof}
First we can show that \(\hat{K}'\) is a stabilizing policy by applying Lemma~\ref{lemma:step_stable}.
Then we can obtain the following bound:
\begin{align*}
\tr{\chi_{\hat{K}'} \hat{L}_{\hat{K}}^T \hat{\Psi} \hat{L}_{\hat{K}}} \leq  \|(\hat{R} + \hat{B}^T \EPhatnext \hat{B})\| \tr{\chi_{\hat{K}'} \hat{L}_{\hat{K}}^T \hat{L}_{\hat{K}}} 
\end{align*}
Combining the above inequality with  Lemma \ref{lemma:smooth},  we can show
\begin{align*}
C(\hat{K}')\! - C(\hat{K})\! &= -4\eta \tr{\!\chi_{\hat{K}'\!} \hat{L}_{\hat{K}}^T \hat{L}_{\hat{K}}}\! + 4\eta^2 \tr{\!\chi_{\hat{K}'} \hat{L}_{\hat{K}}^T \hat{\Psi} \hat{L}_{\hat{K}}\!} \\
	&\leq -2\eta \tr{\chi_{\hat{K}'} \hat{L}_{\hat{K}}^T \hat{L}_{\hat{K}}} \\
	&\leq -2\eta \sigma_{\min}(\chi_{\hat{K}'}) \tr{\hat{L}_K^T \hat{L}_{\hat{K}}} \\
	&\leq -2\eta \mu \tr{\hat{L}_{\hat{K}}^T \hat{L}_{\hat{K}}} \\
	&\leq -\frac{2\eta\mu \sigma_{\min}(\hat{R})}{\| \chi_{\hat{K}^*} \|} \left(C(\hat{K}) - C(\hat{K}^*)\right)
\end{align*}
where the last step follows from Lemma~\ref{lemma:grad_dom}.
\end{proof}

Now we are ready to prove Theorem~\ref{thm:conv}.
\begin{proof} (of Theorem~\ref{thm:conv}, Gauss-Newton case)
Since Lemma~\ref{lemma:one_stepGN} holds for any $\eta \leq \frac{1}{2}$, we have the following contraction at every step:
\begin{align*}
C(\hat{K}^{n+1}) - C(\hat{K}^*) &\leq \left( 1 -  \frac{2\eta\mu}{\| \chi_{\hat{K}^*} \|} \right)  \left( C(\hat{K}^n) - C(\hat{K}^*) \right).
\end{align*}
Then we can obtain the final result using induction.
\end{proof}
\begin{proof} (of Theorem~\ref{thm:conv}, the natural policy gradient case)
By Lemma \ref{lemma:bound}, we have the following bound on the step size:
\begin{align*}
\frac{1}{\| \hat{R} + \hat{B}^T \EPhatnext \hat{B} \|} &\geq \frac{1}{\| \hat{R}\| + \|\hat{B}\|^2 \| \EPhatnext \|} \\
	&\geq \frac{1}{\| \hat{R}\| + \|\hat{B}\|^2 \left( \sum_{i \in \Omega} \| P_i^{\hat{K}} \| \right)}\\
	&\geq \frac{1}{\| \hat{R}\| + \frac{\|\hat{B}\|^2 C(\hat{K})}{\mu}},
\end{align*}
where the second step follows from 
\begin{align*}
\| \EPhatnext \| &\leq \max_{i\in \Omega} \| \EPKnext \| = \max_{i\in \Omega} \| \sum_{j\in \Omega} p_{ij} P_j^{\hat{K}} \| \\
&\qquad \leq \max_{i\in \Omega} \sum_{j\in \Omega} p_{ij} \| P_j^{\hat{K}} \| \leq \sum_{j\in \Omega} \| P_j^{\hat{K}} \|.
\end{align*}
The proof can be completed by induction: For the first step we have \(C(\hat{K}^1) \leq C(\hat{K}^0)\), which is due to Lemma~\ref{lemma:one_step}.
The proof proceeds by arguing that Lemma~\ref{lemma:one_step} can be applied at every step.
If it were the case that \(C(\hat{K}^n) \leq C(\hat{K}^0)\), then
\begin{align*}
\eta &\leq \frac{1}{2\left(\| \hat{R} \| + \frac{\| \hat{B} \|^2 C(\hat{K}^0)}{\mu}\right)} \leq  \frac{1}{2\left(\| \hat{R} \| + \frac{\| \hat{B} \|^2 C(\hat{K}^n)}{\mu}\right)} \\
&\qquad \leq \frac{1}{2\| \hat{R} + \hat{B}^T \hat{\mathcal{E}}(P^{\hat{K}^n}) \hat{B} \|}
\end{align*}
and so we can apply Lemma~\ref{lemma:one_step} to get
\begin{align*}
C(\hat{K}^{n+1}) - C(\hat{K}^*) &\leq \left( 1 -  \frac{2\eta\mu\sigma_{\min}(\hat{R})}{\| \chi_{\hat{K}^*} \|} \right) \times \\
&\qquad \qquad \left( C(\hat{K}^n) - C(\hat{K}^*) \right).
\end{align*}
Obviously, now we have \(C(\hat{K}^{n+1}) \leq C(\hat{K}^0)\) and can repeat the above argument for the next step.
Therefore, the desired conclusion follows from induction.
\end{proof}

\end{document}